\documentclass[a4paper,10pt]{amsart}

\usepackage{hyperref}
\usepackage{amsfonts}
\usepackage{amsmath}
\usepackage{amssymb,amsxtra,amscd,amsthm}
\usepackage{mathrsfs}
\usepackage{todonotes}

\newcommand{\K}{\mathbb{K}}
\newcommand{\R}{\mathbb{R}}
\newcommand{\C}{\mathbb{C}}
\newcommand{\N}{\mathbb{N}}
\newcommand{\Z}{\mathbb{Z}}
\newcommand{\ind}{\text{ind}}
\renewcommand{\span}{\text{span}}

\newtheorem{theorem}{Theorem}[section]
\newtheorem{lemma}[theorem]{Lemma}
\newtheorem{corollary}[theorem]{Corollary}
\newtheorem{proposition}[theorem]{Proposition}

\theoremstyle{definition}
\newtheorem{remark}[theorem]{Remark}
\newtheorem{definition}[theorem]{Definition}

\allowdisplaybreaks

\title{Dynamics of shift operators on non-metrizable sequence spaces}

\author[J.\ Bonet]{Jos\'e Bonet}
\address{Jos\'e Bonet, Institut Universitari de Matem\`atica Pura i Aplicada, Universitat Polit\`ecnica de Val\`encia, 46022 Val\`encia, Spain}
\email{jbonet@mat.upv.es}

\author[T.\ Kalmes]{Thomas Kalmes}
\address{Thomas Kalmes, Chemnitz University of Technology, Faculty of Mathematics, 09107 Chemnitz, Germany}
\email{thomas.kalmes@math.tu-chemnitz.de}

\author[A.\ Peris]{Alfred Peris}
\address{Alfred Peris, Institut Universitari de Matem\`atica Pura i Aplicada, Universitat Polit\`ecnica de Val\`encia, 46022 Val\`encia, Spain}
\email{aperis@mat.upv.es}
\thanks{This article is accepted for publication in Revista Matem\'atica Iberoamericana.}

\begin{document}
	
	\begin{abstract}
		We investigate dynamical properties such as topological transitivity, (sequential) hypercyclicity, and chaos for backward shift operators associated to a Schauder basis on LF-spaces. As an application, we characterize these dynamical properties for weighted generalized backward shifts on K\"othe coechelon sequence spaces $k_p((v^{(m)})_{m\in\N})$ in terms of the defining sequence of weights $(v^{(m)})_{m\in\N}$. We further discuss several examples and show that the annihilation operator from quantum mechanics is mixing, sequentially hypercyclic, chaotic, and topologically ergodic on $\mathscr{S}'(\R)$.
	\end{abstract}

	\maketitle	
	
	\section{Introduction}
	
	The study of dynamical properties of linear operators has attracted much interest in recent years. Most articles concentrate on the dynamics of (continuous linear operators) $T\in L(E)$ defined on a separable Fr\'echet space $E$. The advantage of completeness and metrizability lies in the applicability of Baire category arguments, which are very useful in this context. A few articles deal with dynamics of operators on non-metrizable topological vector spaces (see e.g.\ \cite{Bo00}, \cite{BoFrPeWe2005}, \cite{BoDo12}, \cite{DoKa18}, \cite{GEPe10}, \cite{Kalmes19-3}, \cite{MuPe15}, \cite{Peris18}, \cite{Shkarin2012}, and Chapter 12 in \cite{GEPe11}).
	
	Recall that an operator $T\in L(E)$ on a topological vector space $E$ is called \textit{(topologically) transitive} if for any pair of non-empty, open subsets $U,V\subseteq E$ the set
	$$N_T(U,V)=\{n\in\N;\,T^n(U)\cap V\neq\emptyset\}$$
	is not empty, while $T$ is called \textit{(topologically) mixing}, if these sets are cofinite. More generally, for an infinite subset $I\subseteq\N$, a family $(T_n)_{n\in I}\in L(E)^I$ is called \textit{(topologically) transitive} if for every pair of non-empty, open subsets $U,V\subseteq E$ there is $n\in I$ with $T_n(U)\cap V\neq\emptyset$. Obviously, $T$ is (topologically) mixing if and only if, for any infinite subset $I\subseteq \N$ the family $(T^n)_{n\in I}$ is (topologically) transitive.
	
	Moreover, $T$ is called \textit{(sequentially) hypercyclic} if there is $x\in E$ whose orbit $\{x,Tx,T^2x,\ldots\}$ is (sequentially) dense in $E$. Clearly, every hypercyclic operator is transitive. The converse holds in case $E$ is separable, complete, and metrizable, due to Birkhoff's Transitivity Theorem. Furthermore, a transitive operator $T$ on $E$ is called \textit{chaotic} if the set of periodic points of $T$ is dense in $E$. Finally, $T$ is called \textit{topologically ergodic} if for each pair of non-empty and open subsets $U,V$ of $E$ the set $N_T(U,V)$ is syndetic, i.e.\ there is $p\in\N$ such that $\{n,\ldots,n+p\}$ intersects $N_T(U,V)$ for every $n\in\N$.
	
	The purpose of this article is to characterize dynamical properties for weighted generalized backward shifts on K\"othe coechelon spaces. K\"othe echelon and coechelon spaces play a very relevant role in the theory of Fr\'echet spaces and their applications, for example in connection with the isomorphic classification and the existence of Schauder basis. Moreover, many spaces of analytic or smooth functions are isomorphic to echelon or coechelon spaces.
	We refer the reader to \cite{Bi1988}, \cite{BiBo03}, \cite{BiMeSu1982}, \cite{Vald1982},  \cite{Vogt1983}  and the references therein. Weighted (generalized) backward shifts are natural operators on sequence spaces, and thus, many authors have investigated the above properties of these operators on various sequence spaces (see e.g.\ \cite{Salas1995}, \cite{Grosse-Erdmann2000}, \cite{MaPe2002} and \cite{BeMePePu19}). The paper is organized as follows. In section \ref{LF-sequence spaces} we consider LF-spaces with a special Schauder basis and we study the above dynamical properties for the backward shift associated to these Schauder bases. In section \ref{Koethe coechelon spaces}, on the one hand, we evaluate our results for the special case of K\"othe coechelon spaces $k_p(V)$ and on the other hand we extend them to characterize the above dynamical properties for weighted generalized backward shifts in terms of the defining weight sequence $V=(v^{(m)})_{m\in\N}$. In the final section \ref{examples} we present some examples to illustrate our results and we conclude with some natural open problems. In particular, we consider the special case of dual spaces of power series spaces of infinite type, and as a concrete application we show that the annihilation operator from quantum mechanics is mixing, hypercyclic, chaotic, and topologically ergodic on $\mathscr{S}'(\R)$.
	
	For anything related to functional analysis which is not explained in the text, we refer the reader to \cite{MeVo1997}, and for notions and results about dynamics of linear operators we refer to \cite{BaMa09} and \cite{GEPe11}.
	
	\section{The backward shift on certain LF-spaces}\label{LF-sequence spaces}

	The basic model of linear dynamics in a sequence space is the (unilateral) backward shift
	$$
	B(x_1,x_2,x_3,\dots )=(x_2,x_3,x_4,\dots ).
	$$
	As mentioned in the introduction, several dynamical properties of (weighted and unweighted) backward shifts have been studied on Fr\'echet sequence spaces. It turns out that in certain natural cases one cannot iterate the operator in the space, since each iterate has the range in a bigger space. This is the case, for instance, for the dynamics of the differentiation operator on certain weighted spaces of holomorphic functions (which, at the end, can be represented as the backward shift on a suitable sequence space) studied in \cite{Bo00}, or the ``snake shifts'' introduced in \cite{BoFrPeWe2005}. This is the main motivation to study the dynamics of shift operators on countable inductive limits of Fr\'echet spaces (in short, LF-spaces).
	
	An inductive spectrum of Fr\'echet spaces $(E_m)_{m\in\N}$ is an increasing sequence of Fr\'echet spaces such that the inclusion $E_m \subset E_{m+1}$ is continuous for each $m \in \N$. The inductive limit $E=\ind_m E_m$ of the spectrum is the union of the sequence $(E_m)_{m\in\N}$ and it is endowed with the finest locally convex topology such that the inclusion $E_m \subset E$ is continuous for each $m \in \N$. We assume that the topology of the inductive limit is Hausdorff. This is always the case for K\"othe coechelon spaces. The space $\mathcal{D}(\Omega)$ of test functions for Schwartz distributions is one of the most important examples of an (LF)-spaces. We refer the reader to \cite{Bi1988}, \cite{Vogt1992} and \cite{Weng03} for more information about (LF)-spaces.
	
	In this section we characterize dynamical properties of the backward shift operator on certain LF-spaces.
	
	\begin{definition}
		Let $(E_m)_{m\in\N}$ be an inductive spectrum of Fr\'echet spaces with inductive limit $E=\ind_m E_m$. A sequence $(e_j)_{j\in\N}$ in $E_1$ is called a \textit{stepwise Schauder basis} if $(e_j)_{j\in\N}$ is a Schauder basis for each $E_m, m\in\N$. If the linear mapping on $\text{span}\{e_j;\,j\in\N\}$ defined by $B e_1:=0$ and $B e_j:=e_{j-1}, j\geq 2,$ extends to a continuous linear self-map $B$ on $E$, $B$ is called the \textit{backward shift associated with $(e_j)_{j\in\N}$}.
	\end{definition}
	
	\begin{remark}\label{Grothendieck}
		\begin{enumerate}
			\item[i)] For an LF-space $E=\ind_m E_m$ with stepwise Schauder basis $(e_j)_{j\in\N}$ and associated backward shift $B$ it is an immediate consequence of Grothendieck's Factorization Theorem \cite[Theorem 24.33]{MeVo1997} that for every $m\in\N$ there is $n\in\N$ such that $B(E_m)\subseteq E_n$ and that $B:E_m\rightarrow E_n$ is continuous. By dropping some of the step spaces if necessary we thus may assume without loss of generality that $B:E_m\rightarrow E_{m+1}$, $m\in\N$.
			\item[ii)] The typical example of an LF-space with stepwise Schauder basis $(e_j)_{j\in\N}$ and associated backward shift we have in mind is an LF-sequence space $E=\ind_m E_m$, i.e.\ an LF-subspace $E$ of $\omega=\K^\N$ for which the canonical basis sequence $(e_j)_{j\in\N}$ with $e_j=(\delta_{j,l})_{l\in\N}$ is a Schauder basis in each step space of $E$. If $E$ is invariant under the continuous linear mapping
			$$\omega\rightarrow\omega, (x_j)_{j\in\N}\mapsto (x_{j+1})_{j\in\N}$$
			it follows that its restriction to $E$ has a closed graph, and thus, is a continuous linear self-map of $E$ by de Wilde's Closed Graph Theorem \cite[Theorem 24.31]{MeVo1997}.
		\end{enumerate}
	\end{remark}
	
	We begin with a result which will be used several times within this section.
	
	\begin{proposition}\label{every LF-sequence space is nice}
		Let $E=\ind_m E_m$ be an LF-space with stepwise Schauder basis $(e_j)_{j\in\N}$. Then, for every $m\in\N$, on the Fr\'echet space $E_m$ there is an increasing fundamental sequence of seminorms $(p_k)_{k\in\N}$ satisfying
		$$\forall\,k\in\N, x=\sum_{j=1}^\infty x_j e_j\in E_m, s\in\N:\,p_k(\sum_{j=1}^s x_j e_j)\leq p_k(x).$$
	\end{proposition}
	
	\begin{proof}
		Fix $m\in\N$ and let $(q_k)_{k\in\N}$ be an increasing fundamental system of seminorms for $E_m$. For $s\in\N$ we define
		$$\pi_s:E_m\rightarrow E_m, x=\sum_{j=1}^\infty x_j e_j\mapsto \sum_{j=1}^s x_j e_j.$$
		Then $\{\pi_s;\,s\in\N\}$ is equicontinuous since $(e_j)_{j\in\N}$ is a Schauder basis for $E_m$. In particular, for $x\in E_m$ the set $\{\pi_s(x);\,s\in\N\}$ is a bounded subset of $E_m$ and via
		$$p_k:E_m\rightarrow [0,\infty), x\mapsto \max\{q_k(x), \sup_{s\in\N}q_k(\pi_s(x))\}, k\in\N,$$
		we obtain an increasing fundamental sequence of seminorms $(p_k)_{k\in\N}$ for $E_m$ satisfying the desired property.
	\end{proof}
	
	\begin{proposition}\label{general transitivity}
		Let $E$ be an LF-space with stepwise Schauder basis $(e_j)_{j\in\N}$ and associated backward shift $B$. Then, for an infinite subset $I\subseteq\N$, the following are equivalent.
		\begin{enumerate}
			\item[i)] $(B^n)_{n\in I}$ is transitive on $E$.
			\item[ii)] For each $s\in\N_0$ there are $m\in\N$ and a strictly increasing sequence $(j_k)_{k\in\N}\in I^\N$ such that $\lim_{k\rightarrow\infty}e_{j_k+s}=0$ in $E_m$.
		\end{enumerate}
	\end{proposition}
	
	\begin{proof}
		In order to show that i) implies ii) we assume that $(B^n)_{n\in I}$ is transitive on $E$ but ii) is not satisfied, that is, there is $s\in\N_0$ such that for all $m\in\N$ there is an absolutely convex zero neighborhood $U_m$ in $E_m$ such that $U_m\cap\{e_{j+s};\,j\in I\}$ is finite. By shrinking each $U_m$ if necessary we may assume without loss of generality that $U_m$ and $\{e_{j+s};\,j\in I\}$ are disjoint for each $m\in\N$. Moreover, taking into account Proposition \ref{every LF-sequence space is nice}, we additionally may assume without loss of generality that for each $m\in\N$ there is a continuous seminorm $p^{(m)}$ on $E_m$ satisfying
		\begin{equation}\label{nice seminorms}
		\forall\,x=\sum_{j=1}^\infty x_j e_j\in E_m, r\in\N:\,p^{(m)}(\sum_{j=1}^r x_j e_j)\leq p^{(m)}(x)
		\end{equation}
		such that $\{x\in E_m;\,p^{(m)}(x)\leq 1\}$ and $\{e_{j+s};\,j\in I\}$ are disjoint, i.e.\ $p^{(m)}(e_{j+s})>1$ for all $j\in I,m\in\N$. We first assume $s\geq 1$. 
		
		Since for each $E_m$ the projection onto the span of $e_s$ is continuous, the same holds for $E$ (cf.\ \cite[Proposition 24.7]{MeVo1997}) so that $\{x\in E;|x_s|<1/2\}$ is a zero neighborhood in $E$ as is
		$$W:=\bigcup_{k\in\N}\Big(\sum_{m=1}^k\frac{1}{2^m}\{x\in E_m;\,p^{(m)}(x)\leq 1\}\Big)\cap\{x\in E;\,|x_s|<1/2\}$$
		(cf.\ \cite[Proposition 24.6(c)]{MeVo1997}). From the transitivity of $(B^n)_{n\in I}$ we conclude the existence of $x\in W$ and $n\in I$ with $B^nx\in (3e_s+W)$. In particular, there is $x\in W$ and $n\in I$ with $|(B^nx)_s-3|<1/2$ so that $|x_{n+s}|>5/2$.
		
		As $x\in W$ there are $k\in\N$  and $y^{(m)}\in\{y\in E_m;\,p^{(m)}(y)\leq 1\}, 1\leq m\leq k,$ such that $x=\sum_{m=1}^k\frac{1}{2^m}y^{(m)}$. Thus, applying the projection onto the $(n+s)$ coordinate with respect to the Schauder basis $(e_j)_{j\in\N}$ we get from
		\begin{eqnarray*}
			5/2&<&|x_{n+s}|\leq\sum_{m=1}^k\frac{1}{2^m}|y_{n+s}^{(m)}|<\sum_{m=1}^k\frac{1}{2^m}p^{(m)}(y_{n+s}^{(m)}e_{n+s})\\
			&=&\sum_{m=1}^k\frac{1}{2^m}\Big(p^{(m)}\big(\sum_{j=1}^{n+s}y_j^{(m)}e_j-\sum_{j=1}^{n+s-1} y_j^{(m)}e_j\big)\Big)\\
			&\leq&\sum_{m=1}^k \frac{1}{2^{m-1}}p^{(m)}(y^{(m)})<2,
		\end{eqnarray*}
		the desired contradiction. In particular, there are $m\in\N$ and a strictly increasing sequence $(j_k)_{k\in\N}\in I^\N$ such that $\lim_{k\rightarrow\infty}e_{j_k+1}=0$ in $E_m$. By Remark~\ref{Grothendieck} we find $n\geq m$ such that $B:E_m\to E_n$ and it is continuous. Thus $\lim_{k\rightarrow\infty}e_{j_k}=0$ in $E_n$, and we have shown ii) for the case $s=0$. 
		
		It remains to show that ii) implies transitivity of $(B^n)_{n\in I}$ on $E$. In order to do so, we will show that for every $x,y\in\span\{e_j;\,j\in\N\}$ and each absolutely convex zero neighborhood $W$ in $E$ there are $n\in I$ and $w\in W$ with $B^n(x+w)\in (y+W)$. Since $\span\{e_j;\,j\in\N\}$ is sequentially dense in $E$, transitivity of $(B^n)_{n\in I}$ will follow therefrom.
		
		So, we fix $x,y\in\span\{e_j;\,j\in\N\}$ and an absolutely convex zero neighborhood $W$ in $E$. Let $s\in\N$ be such that $x=\sum_{j=1}^s x_j e_j$ and $y=\sum_{j=1}^s y_j e_j$. Then
		$$\tilde{W}:=\cap_{n=0}^s B^{-n}(W)$$
		is an absolutely convex zero neighborhood in $E$.
		
		Let $(j_k)_{k\in\N}\in I^N$ and $m\in\N$ be as in ii) for $s$. Since $\lim_{k\rightarrow\infty}e_{j_k+s}=0$ in $E_m$, in particular, there is $j_k>s$ (which we fix for the rest of the proof) with
		$$e_{j_k+s}\in\frac{1}{1+\sum_{l=1}^s|y_l|}\tilde{W}$$
		implying
		$$\forall\,0\leq n\leq s:\,B^n(e_{j_k+s})\in\frac{1}{1+\sum_{l=1}^s|y_l|}W.$$
		We define
		$$w:=\sum_{l=1}^s y_l e_{j_k+l}=\sum_{l=1}^s y_l B^{s-l}e_{j_k+s}\in\sum_{l=1}^s \frac{y_l}{1+\sum_{j=1}^s|y_j|}W\subseteq W,$$
		since $W$ is absolutely convex. Moreover, since $j_k>s$
		$$B^{j_k}(x+w)=B^{j_k}w=\sum_{l=1}^s y_l B^{j_k}e_{j_k+l}=\sum_{l=1}^s y_l e_l=y$$
		which proves the claim.
	\end{proof}
	
	The above result enables to characterize transitivity and mixing of backward shifts.
	
	\begin{corollary}\label{transitivity on LF}
		Let $E$ be an LF-space with stepwise Schauder basis $(e_j)_{j\in\N}$ and associated backward shift $B$.
		\begin{itemize}
			\item[a)] The following are equivalent.
			\begin{enumerate}
				\item[i)] $B$ is transitive on $E$.
				\item[ii)] There are $m\in\N$ and a strictly increasing sequence $(j_k)_{k\in\N}\in\N^\N$ such that $\lim_{k\rightarrow\infty}e_{j_k}=0$ in $E_m$.
			\end{enumerate}
			\item[b)] The following are equivalent.
			\begin{enumerate}
				\item[i)] $B$ is topologically mixing on $E$.
				\item[ii)] For every infinite subset $I\subseteq\N$ there are $m\in\N$ and a strictly increasing sequence $(j_k)_{k\in\N}\in I^\N$ such that $\lim_{k\rightarrow\infty}e_{j_k}=0$ in $E_m$.
			\end{enumerate}
		\end{itemize}
	\end{corollary}
	
	\begin{proof}
		Clearly, a) follows immediately from Proposition \ref{general transitivity} applied to $I=\N$. In order to show b), observe that $B$ is mixing if and only if, for every infinite subset $I\subseteq\N$ the family $(B^n)_{n\in I}$ is transitive. By Proposition \ref{general transitivity}, the latter is equivalent to the fact that for every infinite subset $I\subseteq\N$ and each $s\in \N_0$ there are $m\in\N$ and $(j_k)_{k\in\N}\in I^\N$ for which $(e_{j_k+s})_{k\in\N}$ converge to $0$ in $E_m$ which is obviously equivalent to condition ii).
	\end{proof}

	Before we come to a characterization of (sequential) hypercyclicity for backward shifts, we recall that a subset $I\subseteq\N$ is \textit{thick} if
	$$\forall\,p\in\N\,\exists\,j\in\N:\,\{j,j+1,\ldots,j+p\}\subseteq I.$$
	The following criterion for sequential hypercyclicity \cite[Corollary 3]{Peris18} will be crucial for our next result. We include it here for the reader's convenience.
	
	\begin{lemma}\label{Alfred's criterion}
		Let $E$ be a sequentially separable topological vector space and $T\in L(E)$ such that there is a sequentially dense set $E_0:=\{x_n:\,n\in\N\}\subset E$, a sequence of maps $S_n:E_0\rightarrow E, n\in\N$, a subspace $Y\subset E$ with a finer topology $\tau$ such that $(Y,\tau)$ is an F-space, and an increasing sequence $(n_k)_{k\in\N}$ of natural numbers ($n_0:=0$) satisfying:
		\begin{enumerate}
			\item[i)] For all $j\in\N$ and each $x\in E_0$ there is $l\in\N$ such that $(T^{n_k}S_{n_j}x)_{k\geq l}\subset Y$ and converges to 0 in $(Y,\tau)$,
			\item[ii)] for all $j\geq 0$ and each $x\in E_0$ there is $l\in\N$ such that $(T^{n_j}S_{n_k}x)_{k\geq l}\subset Y$ and converges to 0 in $(Y,\tau)$,
			\item[iii)] for each $x\in E_0$ there is $l\in\N$ such that $(x-T^{n_k}S_{n_k}x)_{k\geq l}\subset Y$ and converges to 0 in $(Y,\tau)$.
		\end{enumerate}
		Then $T$ is sequentially hypercyclic.
	\end{lemma}
	
	\begin{proposition}\label{hypercyclicity on LF}
		Let $E$ be an LF-space with stepwise Schauder basis $(e_j)_{j\in\N}$ and associated backward shift $B$. Then, the following are equivalent.
		\begin{enumerate}
			\item[i)] $B$ is sequentially hypercyclic on $E$.
			\item[ii)] $B$ is hypercyclic on $E$.
			\item[iii)] There are $m\in\N$ and a thick set $I\subseteq\N$ such that $\displaystyle{\lim_{I\ni j\rightarrow\infty}e_j=0}$ in $E_m$.
		\end{enumerate}
	\end{proposition}
	
	\begin{proof}
		Trivially, i) implies ii). In order to show that ii) implies iii) we define for $p\in\N$
		$$\pi_p:E\rightarrow \K^p, x=\sum_{j=1}^\infty x_j e_j\mapsto (x_1,\ldots,x_p)$$
		which is surjective. Let $x\in E$ be a hypercyclic vector for $B$ and let $m\in\N$ be such that $x\in E_m$. Then, for every $r\in\N$ the set $\{\pi_p(B^nx); n\geq r\}$ is dense in $\K^p$. In particular, for every $p,r\in\N$ there is $n_p^r\geq r$ such that
		\begin{equation*}
		1/2>\max_{1\leq l\leq p}|\big(\pi_p(B^{n_p^r}x)\big)_l-1|.
		\end{equation*}
		Hence, for $p,r\in\N$ there is $n_p^r\geq r$ such that $|x_{n_p^r+l}|\geq 1/2$ for all $1\leq l\leq p$ which implies
		\begin{equation}\label{constructing thick set}
		\exists\,(n_p)_{p\in\N}\in\N^\N\text{ strictly increasing }\forall\,1\leq l\leq p:\,|x_{n_p+l}|\geq 1/2.
		\end{equation}
		Obviously, $I:=\{n_p+l;\,p\in\N, 1\leq l\leq p\}$ is a thick subset of $\N$. Since $x\in E_m$ and since $(e_j)_{j\in\N}$ is a Schauder basis of $E_m$ the sequence $(x_j e_j)_{j\in\N}$ converges to $0$ in $E_m$, where as usual $x=\sum_{j=1}^\infty x_j e_j$. Thus, for an arbitrary absolutely convex zero neighborhood $U$ in $E_m$ there is $N\in\N$ such that $x_j e_j\in U$ whenever $j\geq N$. Hence, for all $p\in\N$ with $n_p>N$ and each $1\leq l\leq p$ we conclude by the absolute convexity of $U$ and (\ref{constructing thick set})
		$$e_{n_p+l}=\frac{1}{x_{n_p+l}}x_{n_p+l}e_{n_p+l}\in 2U$$
		which proves $\lim_{I\ni j\rightarrow\infty}e_j=0$ in $E_m$.
		
		It remains to show that iii) implies i). To accomplish this we introduce
		$$S:\span\{e_j;\,j\in\N\}\rightarrow\span\{e_j;j\in\N\},\sum_{j=1}^s x_j e_j\mapsto\sum_{j=1}^s x_j e_{j+1}$$
		as well as $S_n:=S^n$. Note that $E_0:=\span\{e_j;\,j\in\N\}$ is a sequentially dense subspace of $E$ and that for $k,l\in\N$
		$$B^k S_l (\sum_{j=1}^s x_j e_j)=\sum_{j=1}^s x_j\, e_{\max\{0, l-k+j\}},$$
		where $e_0:=0$.
		
		Let $I=\{n_k+l;\,k\in\N, 1\leq l\leq k\}$ with $(n_k)_{k\in\N}$ strictly increasing and $m\in\N$ be such that $\lim_{I\ni j\rightarrow\infty}e_j=0$ in $E_m$. We define recursively $\tilde{n}_1:=n_1$ and for $k\in\N$
		$$\tilde{n}_{k+1}:=\sum_{r=1}^k\tilde{n}_r+n_{k+1+\sum_{r=1}^k\tilde{n}_r}.$$
		Then, for every $k,j\in\N$ with $k>j$ and each $1\leq l\leq k$ we have
		\begin{eqnarray}\label{useful inclusion}
		\tilde{n}_k-\tilde{n}_j+l&=&\sum_{r=1, r\neq j}^{k-1}\tilde{n}_r+n_{k+\sum_{r=1}^{k-1}\tilde{n}_r}+l\\
		&\in&\{n_{k+\sum_{r=1}^{k-1}\tilde{n}_r}+1,\ldots,n_{k+\sum_{r=1}^{k-1}\tilde{n}_r}+k+\sum_{r=1}^{k-1}\tilde{n}_r\}.\nonumber
		\end{eqnarray}
		Next we fix $x=\sum_{j=1}^s x_j e_j\in\span\{e_j;\,j\in\N\}$. Given an absolutely convex zero neighborhood $U$ in $E_m$ there is $K\in\N$ such that $e_{n_k+l}\in U$ whenever $k\geq K$ and $1\leq l\leq k$. From (\ref{useful inclusion}) we obtain in particular that whenever $k$ is such that $k+\sum_{r=1}^{k-1}\tilde{n}_r\geq K$
		$$\forall j<k, 1\leq l\leq k: e_{\tilde{n}_k-\tilde{n}_j+l}\in U$$
		so that for all $k$ with $k+\sum_{r=1}^{k-1}\tilde{n}_r\geq K$ it follows
		$$\forall\,j<k:\,B^{\tilde{n}_j}S_{\tilde{n}_k}\Big(\sum_{l=1}^sx_l e_l\Big)=\sum_{l=1}^sx_le_{\tilde{n}_k-\tilde{n}_j+l}\in(1+\sum_{l=1}^s|x_l|)U$$
		by the absolute convexity of $U$. Hence, $(B^{\tilde{n}_j}S_{\tilde{n}_k}x)_{k\in\N}$ converges to $0$ in $E_m$. Since trivially $(B^{\tilde{n}_k}S_{\tilde{n}_j}x)_{k\in\N}$ and $(x-B^{\tilde{n}_k}S_{\tilde{n}_k}x)_{k\in\N}$ both converge to $0$ in $E_m$ it follows form Lemma \ref{Alfred's criterion} that $B$ is sequentially hypercyclic on $E$.
	\end{proof}
	
	Before we come to the next result of this section, we recall some notions for subsets of $\N$. Recall that for a given $m\in\N$ a subset $A\subseteq\N$ is \textit{$m$-syndetic} if $\N\subseteq\{a-k;\,a\in A, k\in\{0,\ldots,m\}\}$. In case that $A$ is $m$-syndetic for some $m$, we simply say that $A$ is \textit{syndetic}. Given $m\in\N$, a set $A\subseteq\N$ is called \textit{piecewise $m$-syndetic} if $A=A_1\cap A_2$ with $A_1\subseteq\N$ thick and $A_2\subseteq\N$ $m$-syndetic. \textit{Piecewise syndetic} sets are those which are piecewise $m$-syndetic for some $m\in\N$. Finally, given $n,m\in\N$ with $n>m$, we say that a finite set $F\subseteq\N$ is \textit{$(n,m)$-syndetic} if $F=J\cap A$, with $J\subseteq\N$ being an interval of length $n$ and $A\subseteq\N$ $m$-syndetic.
	
	\begin{proposition}\label{mixing implies hypercyclic}
		Let $E$ be an LF-space with stepwise Schauder basis $(e_j)_{j\in\N}$ and associated backward shift $B$. If $B$ is mixing then $B$ is sequentially hypercyclic.
	\end{proposition}
	
	\begin{proof}
		By Remark \ref{Grothendieck} we can assume without loss of generality that for $E=\ind_m E_m$ we have
		\begin{equation}\label{B acts step by step}
		\forall\,m\in\N:\,B:E_m\rightarrow E_{m+1}\text{ continuously}.
		\end{equation}
		By Proposition \ref{hypercyclicity on LF}, we have to show the existence of $m\in\N$ and a thick set $I\subseteq\N$ such that $\lim_{I\ni j\rightarrow\infty}e_j=0$ in $E_m$. Actually, if we find $m,\tilde{m} \in\N$ and a piecewise $\tilde{m}$-syndetic set $\tilde{I}\subseteq\N$ such that $\lim_{\tilde{I}\ni j\rightarrow\infty}e_j=0$ in $E_m$ then by (\ref{B acts step by step}) and the continuity of the inclusions $E_k\hookrightarrow E_{k+1}, k\in\N,$ we immediately conclude the existence of a thick set $I\subseteq\N$ such that $\lim_{I\ni j\rightarrow\infty}e_j=0$ in $E_{m+\tilde{m}}$, so that $B$ is indeed sequentially hypercyclic.
		
		Let us assume that for every $m\in\N$ there is no piecewise syndetic $\tilde{I}\subseteq\N$ with $\lim_{\tilde{I}\ni j\rightarrow\infty}e_j=0$ in $E_m$. We will show that under this assumption $B$ cannot be mixing.
		
		\textit{Claim 1:} For every piecewise syndetic $I\subseteq\N$ and for every $m\in\N$, there are a piecewise syndetic subset $I_m\subseteq I$ as well as a zero neighborhood $U_m$ in $E_m$ such that $U_m\cap\{e_j;\,j\in I_m\}=\emptyset$.
		
		In order to prove claim 1, let $I\subseteq\N$ be piecewise $k$-syndetic for some $k\in\N$ and let $m\in\N$. Moreover, let $(U_n)_{n\in\N}$ be a decreasing zero neighborhood basis in $E_m$. Then, there is $n_0>k$ such that for all $(n_0,k)$-syndetic sets $F$ we have $\{e_j;\,j\in F\}\cap (E_m\backslash U_{n_0})\neq\emptyset$. Indeed, otherwise we had that for each $n>k$ there was a $(n,k)$-syndetic sets $F_n$ with $\{e_j;\,j\in F_n\}\subseteq U_n$. Since $\tilde{I}:=\cup_{n>k}F_n$ is piecewise $k$-syndetic and for every $l\in\N, l>k,$ we have $e_j\in U_l$ whenever $j\in\cup_{n\geq l}F_n$, this would imply $\lim_{\tilde{I}\ni j\rightarrow\infty}e_j=0$ in $E_m$ which contradicts our assumption that there is no piecewise syndetic $\tilde J\subseteq\N$ with $\lim_{\tilde J\ni j\rightarrow\infty}e_j=0$ in $E_m$.
		
		Since $I$ is piecewise $k$-syndetic, for each $n>n_0$ we find a $(n\cdot n_0,k)$-syndetic set $F_n\subseteq I$. We write as a disjoint union
		$$F_n=\bigcup_{i=1}^n F_{n,i},$$
		where each $F_{n,i}$ is a $(n_0,k)$-syndetic set, $i=1,\ldots, n$. Hence, for each $i=1,\ldots,n$ there is $j(n,i)\in F_{n,i}$ such that $e_{j(n,i)}\notin U_{n_0}$. The set $\tilde{F}_n:=\{j(n,i);\,i=1\,\ldots,n\}$ is clearly a $(n\,n_0,2n_0)$-syndetic set. We further define
		$$I_m:=\bigcup_{n=n_0+1}^\infty \tilde{F}_n\subseteq I$$
		which is piecewise $2n_0$-syndetic and
		$$\forall\,j\in I_m:\,e_j\notin U_{n_0}$$
		which proves claim 1.
		
		\textit{Claim 2:} There exist a decreasing sequence $(I_m)_{m\in\N}$ of piecewise syndetic sets $I_m\subseteq\N$ and a sequence $(U_m)_{m\in\N}$ of zero neighborhoods $U_m$ in $E_m$, $m\in\N$, such that
		$$\forall\,m\in\N:\,U_m\cap\{e_j;\,j\in I_m\}=\emptyset.$$
		Indeed, proceeding by induction, we obtain claim 2 immediately from claim 1.
		
		Now, let $(I_m)_{m\in\N}$ and $(U_m)_{m\in\N}$ be as in claim 2. We select an increasing sequence $(j_m)_{m\in\N}\in\prod_{m\in\N}I_m$ and set $I:=\{j_m;\,m\in\N\}$ which is an infinite set such that for every $m\in\N$ the zero neighborhood $U_m$ in $E_m$ is disjoint to $\{e_j;\,j\in I, j\geq m\}$. Hence, by Corollary \ref{transitivity on LF} b), $B$ is not mixing on $E$.
	\end{proof}
	
	The last result of this section gives a sufficient condition under which the backward shift is topologically ergodic. However, this sufficient condition is not necessary, in general, as is shown in Proposition~\ref{mixing but not topologically ergodic condition} below.
	
	\begin{proposition}\label{ergodic on LF}
		Let $E$ be an LF-space with stepwise Schauder basis $(e_j)_{j\in\N}$ and associated backward shift $B$. Assume that there is $m\in\N$ such that for every zero neighborhood $W$ in $E_m$ the set $I_W:=\{j\in\N;\,e_j\in W\}$ is syndetic. Then $B$ is topologically ergodic on $E$.
	\end{proposition}
	
	\begin{proof}
		Let $U,V\subseteq E$ be open and non-empty. Then, there are $x,y\in\text{span}\{e_j;j\in\N\}$ and an absolutely convex zero neighborhood $\tilde{W}$ in $E$ such that
		$$x+\tilde{W}\subseteq U\mbox{ and }y+\tilde{W}\in V.$$
		Let $s\in\N$ be such that $x=\sum_{j=1}^s x_j e_j$ as well as $y=\sum_{j=1}^s y_j e_j$ and define
		$$\tilde{W}_1:=\cap_{n=0}^s B^{-n}(\frac{1}{1+\sum_{l=1}^s|y_l|}\tilde{W})$$
		which is a zero neighborhood in $E$. Denoting by $i_m$ the canonical injection of $E_m$ into $E$, $W:=i_m^{-1}(\tilde{W}_1)$ is a zero neighborhood in $E_m$ so that by the hypothesis $I_W:=\{j\in\N;\,e_j\in W\}$ is syndetic. From the definition it follows
		$$\forall\,n=0,\ldots, s, j\in I_W:\,B^ne_j\in\frac{1}{1+\sum_{l=1}^s|y_l|}\tilde{W}.$$
		For $j\in I_W\cap\{s+1,s+2,\ldots\}$ we set
		$$w_j:=\sum_{k=1}^s y_k e_{j-s+k}=\sum_{k=1}^s y_k B^{s-k}e_j\in \tilde{W}.$$
		Then, for $j\in I_W\cap\{2s+1,2s+2,\ldots\}$ it holds
		$$B^{j-s}(x+w_j)=B^{j-s}w_j=\sum_{k=1}^s y_k B^{j-s}e_{j-s+k}=y$$
		so that $B^{j-s}(x+\tilde{W})\cap (y+\tilde{W})\neq\emptyset$ which implies the proposition since with $I_W$ also $\{n;\,n=j-s, j\in I_W, j>2s\}$ is syndetic.
	\end{proof}
	
	\section{The backward shift on K\"othe coechelon spaces}\label{Koethe coechelon spaces}
	
	In this section we evaluate and complement the results from the previous section for the case of K\"othe coechelon spaces. Let $V=(v^{(m)})_{m\in\N}$ be a decreasing sequence of strictly positive weights on $\N$, i.e.\ $v^{(m)}=(v^{(m)}_j)_{j\in\N}\in (0,\infty)^\N, m\in\N,$ such that
	$$\forall m,j\in\N:\,v^{(m)}_j\geq v^{(m+1)}_j.$$
	For $m\in\N$ and $1\leq p<\infty$ we define as usual
	$$\ell_p(v^{(m)}):=\{x=(x_j)_{j\in\N}\in\omega;\,(x_j v_j^{(m)})_{j\in\N}\in\ell_p\}.$$
	Equipped with the norm $\|x\|_{p,v^{(m)}}:=\|(x_j v_j^{(m)})_{j\in\N}\|_{\ell_p}$ this is a Banach space for which $(e_j)_{j\in\N}=((\delta_{j,l})_{l\in\N})_{j\in\N}$ is a Schauder basis. Due to the fact that $(v^{(m)})_{m\in\N}$ is decreasing, $(\ell_p(v^{(m)}))_{m\in\N}$ is an inductive spectrum of Banach spaces whose inductive limit we denote by $k_p(V)$ or $k_p((v^{(m)})_{m\in\N})$. Analogously, for $m\in\N$ we set
	$$c_0(v^{(m)}):=\{x=(x_j)_{j\in\N}\in\omega;\,(x_j v^{(m)}_j)_{j\in\N}\in c_0\}$$
	which is Banach space when equipped with the norm $\|x\|_{0,v^{(m)}}:=\sup_{j\in\N}|x_j|v^{(m)}_j$. Again, $(e_j)_{j\in\N}=((\delta_{j,l})_{l\in\N})_{j\in\N}$ is a Schauder basis of $c_0(v^{(m)})$ and $(c_0(v^{(m)}))_{m\in\N}$ is an inductive spectrum of Banach spaces whose inductive limit we denote by $k_0(V)$ or $k_0((v^{(m)})_{m\in\N})$. In particular, $k_p(V), p\in\{0\}\cup[1,\infty),$ is an LF-space for which the standard basis sequence $(e_j)_{j\in\N}=((\delta_{j,l})_{l\in\N})_{j\in\N}$ is a stepwise Schauder basis. In this section, $(e_j)_{j\in\N}$ always stands for this basis sequence.
	
	For a given decreasing sequence of weights $V=(v^{(m)})_{m\in\N}$ we denote by $\bar{V}$ the associated family of weights of its projective description, i.e.\ for $\bar{v}=(\bar{v}_j)_{j\in\N}\in [0,\infty)^\N$
	$$\bar{v}\in\bar{V}\Leftrightarrow\forall\,m\in\N\,\exists\,\alpha_m>0\,\forall\,j\in\N:\,\bar{v}_j\leq \alpha_m v^{(m)}_j.$$
	Then $k_p(V)=K_p(\bar{V}), p\in\{0\}\cup[1,\infty)$ where
	$$K_p(\bar{V}):=\text{proj}_{\bar{v}\in\bar{V}}\ell_p(\bar{v}), 1\leq p<\infty,$$
	respectively
	$$K_0(\bar{V}):=\text{proj}_{\bar{v}\in\bar{V}}c_0(\bar{v}),$$
	see \cite{BiMeSu1982}.
	
	Before we evaluate the results from the previous section in the particular context of K\"othe coechelon spaces, we characterize, when the backward shift $B$ associated with the standard basis sequence is well-defined (and continuous) on $k_p(V)$.
	
	\begin{proposition}\label{continuity of the backward shift on Koethe coechelon spaces}
		Let $V=(v^{(m)})_{m\in\N}$ be a decreasing sequence of strictly positive weights and $p\in\{0\}\cup[1,\infty)$. Then the following are equivalent.
		\begin{itemize}
			\item[i)] The backward shift $B:k_p(V)\rightarrow k_p(V)$ is well-defined.
			\item[ii)] The backward shift $B:k_p(V)\rightarrow k_p(V)$ is continuous.
			\item[iii)] For every $m\in\N$ there exist $n\in\N, n\geq m,$ and $C>0$ such that
			$$\forall\,j\in\N: v^{(n)}_j\leq C v^{(m)}_{j+1}.$$
		\end{itemize}
	\end{proposition}
	\begin{proof}
		Clearly, iii) implies ii), and i) follows from ii). Moreover, by Remark \ref{Grothendieck} ii), ii) follows from i).
		
		Finally, if ii) holds it follows from Grothendieck's Factorization Theorem \cite[Theorem 24.33]{MeVo1997} that for each $m\in\N$ there is $n\in\N$ such that
		$$B:\ell_p(v^{(m)})\rightarrow\ell_p(v^{(n)}), (x_j)_{j\in\N}\mapsto (x_{j+1})_{j\in\N}$$
		is well-defined and continuous - in case $1\leq p<\infty$, analogously for $p=0$ - so that there is $C>0$ such that
		$$\forall\,j\in\N:\,v^{(n)}_j=\|e_j\|_{p,v^{(n)}}\leq C\|e_{j+1}\|_{p,v^{(m)}}=Cv^{(m)}_{j+1},$$
		i.e.\ iii) is true.
	\end{proof}
	
	It should be noted that continuity of the backward shift $B$ (and being well-defined) on $k_p(V)$ is independent of $p$.
	
	For a decreasing sequence of strictly positive weights $V$ we denote by $A(V)=(a^{(m)})_{m\in\N}$ the K\"othe matrix where $a^{(m)}:=1/v^{(m)}$ (see e.g.\ \cite[Chapter 27]{MeVo1997} for the notion of a K\"othe matrix as well as for the corresponding K\"othe (echelon) sequence spaces $\lambda_2(A(V))$ appearing in the next theorem). Recall that a continuous linear operator between locally convex spaces is called \textit{Montel} if it maps bounded sets to relatively compact sets.
	
	\begin{theorem}\label{dynamics of backward shift on Koethe coechelon spaces}
		Let $V=(v^{(m)})_{m\in\N}$ be a decreasing sequence of strictly positive weights such that the backward shift $B$ is continuous on $k_p(V), p\in\{0\}\cup [1,\infty)$.
		\begin{enumerate}
			\item[a)] The following are equivalent.
			\begin{enumerate}
				\item[i)] $B$ is transitive on $k_p(V)$.
				\item[ii)] There is $m\in\N$ such that $\liminf_{j\rightarrow\infty}v^{(m)}_j=0$.
			\end{enumerate}
			\item[b)] The following are equivalent.
			\begin{enumerate}
				\item[i)] $B$ is sequentially hypercyclic on $k_p(V)$.
				\item[ii)] $B$ is hypercyclic on $k_p(V)$.
				\item[iii)] There are $m\in\N$ and a thick set $I\subseteq\N$ such that $\lim_{I\ni j\rightarrow\infty}v^{(m)}_j=0$.
			\end{enumerate}
			\item[c)] The following are equivalent.
			\begin{enumerate}
				\item[i)] $B$ is topologically mixing on $k_p(V)$.
				\item[ii)] For every infinite set $I\subseteq \N$ there is $m\in\N$ with $\liminf_{I\ni j\rightarrow\infty} v^{(m)}_j=0$.
				\item[iii)] For every $\bar{v}\in\bar{V}$ we have $\lim_{j\rightarrow\infty}\bar{v}_j=0$.
				\item[iv)] The natural map $i:\ell_2\rightarrow k_2(V)$ is (well-defined and) compact.
				\item[v)] The natural map $i:\lambda_2(A(V))\rightarrow\ell_2$ is (well-defined and) Montel.
			\end{enumerate}
			\item[d)] Assume that there is $m\in\N$ such that the set $I_\varepsilon:=\{j\in\N;\,v^{(m)}_j<\varepsilon\}$ is syndetic for every $\varepsilon>0$. Then $B$ is topologically ergodic on $k_p(V)$.
		\end{enumerate}
	\end{theorem}
	
	\begin{proof}
		Since for each $m,j\in\N$ we have $\|e_j\|_{p,v^{(m)}}=v^{(m)}_j$, part a) follows immediately from Corollary \ref{transitivity on LF} a), part b) follows from Proposition \ref{hypercyclicity on LF}, and part d) is a direct consequence of Proposition \ref{ergodic on LF}. Moreover, that i) and ii) in part c) are equivalent is an immediate consequence of Corollary \ref{transitivity on LF} b).
		
		Next, we assume that ii) of c) holds but that there is $\bar{v}\in\bar{V}$ which does not converge to $0$. Hence, there are $\varepsilon>0$ and a strictly increasing sequence $(j_k)_{k\in\N}$ in $\N$ such that
		$$\forall\,k\in\N:\,\bar{v}_{j_k}>\varepsilon.$$
		Choose $m\in\N$ according to c) ii) for $I:=\{j_k;\,k\in\N\}$, i.e.\ $\inf_{k\in\N}v^{(m)}_{j_k}=0$. As $\bar{v}\in\bar{V}$ there is $\alpha_m>0$ such that $\bar{v}_j\leq\alpha_m v^{(m)}_j$ for every $j\in\N$. In particular,
		$$\forall\,k\in\N:\,\varepsilon<\bar{v}_{j_k}\leq\alpha_m v^{(m)}_{j_k}$$
		contradicting $\inf_{k\in\N}v^{(m)}_{j_k}=0$. Thus, c) ii) implies c) iii).
		
		In order to prove the converse implication, assume that c) iii) holds but that for some infinite $I\subseteq\N$ for each $m\in\N$ there is $\varepsilon_m>0$ such that $v^{(m)}_j\geq\varepsilon_m$ whenever $j\in I$. Then, via
		$$\bar{v}_j:=\inf_{m\in\N}\frac{v^{(m)}_j}{\varepsilon_m}, j\in\N,$$
		we obtain $\bar{v}\in\bar{V}$ with $\bar{v}_j\geq 1$ for every $j\in I$. Since $I\subseteq\N$ is supposed to be infinite, this contradicts c) iii), so that c) iii) implies c) ii).
		
		So far we have shown that i), ii), and iii) in c) are equivalent. Moreover, c) iii) holds if and only if the inclusion $\ell_2\hookrightarrow\ell_2(\bar{v})$ is well-defined and compact for all $\bar{v}\in\bar{V}$. Hence, c) iii) implies that (by Tychonov's Theorem) the natural map $i:\ell_2\rightarrow k_2(V)=K_2(\bar{V})$ is well-defined and compact. On the other hand, if
		$$i:\ell_2\rightarrow k_2(V)=K_2(\bar{V})=\text{proj}_{\bar{v}\in\bar{V}}\ell_2(\bar{v})$$
		is compact, it follows that $\ell_2\hookrightarrow\ell_2(\bar{v}), \bar{v}\in\bar{V},$ is compact. Thus, c) iii) and c) iv) are equivalent.
		
		Finally, taking into account that $\lambda_2(A(V))$ is the strong dual of $k_2(V)$ (see e.g.\ \cite[Proposition 27.3, Proposition 27.13]{MeVo1997}) it follows from \cite[Corollary 2.4]{DiDo1993} that c) iv) and c) v) are equivalent.
	\end{proof}
	
	Next, we give a characterization of when the backward shift is chaotic on K\"othe coechelon spaces.
	
	\begin{proposition}\label{chaos for Koethe coechelon spaces}
		Let $V=(v^{(m)})_{m\in\N}$ be a decreasing sequence of strictly positive weights such that the backward shift $B$ is continuous on $k_p(V), p\in\{0\}\cup[1,\infty)$. Then, the following are equivalent.
		\begin{enumerate}
			\item[i)] $B$ has a periodic point $x\in k_p(V), x\neq 0$.
			\item[ii)] There is $m\in\N$ such that $v^{(m)}\in\ell_p$, respectively, $v^{(m)}\in c_0$ when $p=0$.
			\item[iii)] $B$ is chaotic, mixing, and sequentially hypercyclic on $k_p(V)$.
			\item[iv)] $B$ is chaotic on $k_p(V)$.
		\end{enumerate}
		In particular, $B$ is sequentially hypercyclic whenever $B$ is chaotic.
	\end{proposition}
	
	\begin{proof}
		Trivially, iii) implies iv) and iv) implies i). We show that i) implies ii). Thus, let $x\in k_p(V)\backslash\{0\}$ be periodic for $B$. We choose $m_0,N,j_0\in\N$ such that $x\in \ell_p(v^{(m_0)})$ (we consider the case $1\leq p<\infty$; the case $p=0$ is analogous), $B^Nx=x$, and $x_{j_0}\neq 0$. Then $x$ is a periodic sequence with period $N$ and $(x_{j_0+jN})_{j\in\N}$ is a constant non-null sequence. Since
		$$|x_{j_0}|^p\sum_{j=1}^\infty |v^{(m_0)}_{j_0+jN}|^p\leq\|x\|_{p,v^{(m_0)}}^p<\infty$$
		it follows that
		$$\sum_{j=1}^\infty (v^{(m_0)}_{j_0+jN})^p<\infty.$$
		We can find $m>m_0$ such that $B^nx\in\ell_p(v^{(m)})$ for every $n\in\{1,\ldots, N-1\}$ and applying the arguments from above to $B^n x$, $n=1\ldots,N-1$ we get
		$$\forall\,n\in\{1,\ldots,N-1\}:\,\sum_{j=1}^\infty (v^{(m)}_{j_0+jN-n})^p<\infty$$
		which implies $v^{(m)}\in\ell_p$.
		
		Next, if ii) holds, it follows from Theorem \ref{dynamics of backward shift on Koethe coechelon spaces} c) that $B$ is mixing, and thus sequentially hypercyclic, too, by Proposition \ref{mixing implies hypercyclic}. We define
		$$H:=\{x=(x_j)_{j\in\N}\in\omega;\,x\text{ periodic}\}.$$
		With $m$ as in ii) and $v^{(m)}\in\ell_p$ it follows that $\ell_\infty\subseteq \ell_p(v^{(m)})$, hence $H\subseteq\ell_p(v^{(m)})\subseteq k_p(V)$. Clearly, every $x\in H$ is periodic for $B$, so it is enough to show that $H$ is dense in $\ell_p(v^{(m)})$ (which is dense in $k_p(V)$ since $\text{span}\{e_j;\,j\in\N\}$ is). The latter will be accomplished once we have shown $e_k\in\overline{H}^{\ell_p(v^{(m)})}, k\in\N$. So, we fix $k\in\N$ and $\varepsilon>0$. Select $i\in\N, i>k,$ such that $\sum_{j=i+1}^\infty (v^{(m)}_j)^p<\varepsilon^p$ and set
		$$z:=\sum_{j=0}^\infty e_{k+ji}\in H.$$
		Then
		$$\|e_k-z\|_{p,v^{(m)}}^p=\sum_{j=1}^\infty (v^{(m)}_{k+ji})^p<\varepsilon^p,$$
		so that indeed $e_k\in\overline{H}^{\ell_p(v^{(m)})}$. Hence, ii) implies iii).
	\end{proof}
	
	In the remainder of this section we will generalize the results for the backward shift to weighted generalized backward shifts. In order to do so, we first introduce some terminology.
	
	\begin{definition}
		A \textit{symbol $\psi$} is a bijection $\psi:\N\rightarrow\N\backslash\{1\}$ satisfying
		$$\N=\{1\}\cup\bigcup_{n\in\N}\{\psi^{n}(1)\}.$$
		Moreover, a \textit{weight (sequence) $w$} is a sequence $w=(w_j)_{j\in\N}\in\omega$ such that $w_j\neq 0, j\in\N$.
		Then,
		$$
		B_{w,\psi}:\omega\rightarrow\omega,x=(x_j)_{j\in\N}\mapsto(w_{\psi(j)}x_{\psi(j)})_{j\in\N}
		$$
		is called the \textit{weighted generalized backward shift} (with weight sequence $w$ and symbol $\psi$). In case $w_j=1, j\in\N,$ we simply write $B_\psi$ (\textit{generalized backward shift}) and in case $\psi(j)=j+1$, we write $B_w$ instead of $B_{w,\psi}$ (\textit{weighted backward shift}). Actually, $B_{w,\psi}=C_\psi \circ D_w$, a weighted composition operator, where the composition operator with symbol $\psi$ is defined as $C_\psi((x_j)_j)=(x_{\psi(j)})_j$, and the multiplication (diagonal) operator with weight $w$ is $D_w((x_j)_j)=(w_jx_j)_j$.
	\end{definition}

	\begin{proposition}\label{conjugacy}
		Let $\psi$ be a symbol and $w$ a weight sequence. Then
		$$
		T_{w,\psi}:\omega\rightarrow\omega, x=(x_j)_{j\in\N}\mapsto ((\prod_{l=0}^{j-1} w_{\psi^{l}(1)} )x_{\psi^{j-1}(1)})_{j\in\N}
		$$
		is an isomorphism such that $T_{w,\psi}^{-1}\circ B \circ T_{w,\psi}=B_{w,\psi}$. Here, as usual $\psi^{0}(1):=1$.
	\end{proposition}
	
	\begin{proof}
		Since $\psi:\N\rightarrow\N\backslash\{1\}$ is a symbol, $\N=\{1\}\cup\bigcup_{n\in\N}\psi^{n}(1)$ and this union is a partition of $\N$. Hence,
		$$
		\chi:\N\rightarrow\N,\chi(1):=1, \ \ \chi(j+1):=\psi^{j}(1), \ \ j\in \N ,
		$$
		is a bijection. Clearly,
		\begin{equation}\label{conjugacy 2}
		\forall\,j\in\N:\,\psi (\chi(j))=\chi(j+1).
		\end{equation}
		With this, one readily sees
		$$
		\forall\,x\in\omega:\,T_{w,\psi}\,x=((\prod_{l=1}^j w_{\chi(l)})x_{\chi(j)})_{j\in\N}
		$$
		which implies that $T_{w,\psi}$ is bijective. Obviously, $T_{w,\psi}$ is also bicontinuous. Finally, for $x\in\omega$ we have
		\begin{eqnarray*}
			T_{w,\psi}\big(B_{w,\psi}\, x\big)&=&T_{w,\psi}\big((w_{\psi(j)} x_{\psi(j)})_{j\in\N}\big)=\big((\prod_{l=1}^j w_{\chi(l)}) w_{\psi(\chi(j))} x_{\psi(\chi(j))}\big)_{j\in\N}\\
			&=&\big((\prod_{l=1}^{j+1} w_{\chi(l)}) x_{\chi(j+1)}\big)_{j\in\N}=B\big(T_{w,\psi}\,x\big),
		\end{eqnarray*}
		where we have used (\ref{conjugacy 2}) in the third equality. Since $T_{w,\psi}$ is bijective, the claim follows.
	\end{proof}
	
	\begin{corollary}\label{continuity of generalized backward shifts on Koethe coechelon spaces}
		Let $\psi$ be a symbol and $w$ a weight sequence. Moreover, let $V=(v^{(m)})_{m\in\N}$ be a decreasing sequence of strictly positive weights and $p\in\{0\}\cup[1,\infty)$. Then, the following are equivalent.
		\begin{enumerate}
			\item[i)] $k_p(V)$ is invariant under $B_{w,\psi}$.
			\item[ii)] $B_{w,\psi}:k_p(V)\rightarrow k_p(V)$ is well-defined and continuous.
			\item[iii)] For every $m\in\N$ there are $n\in\N$ and $C>0$ such that
			$$\forall\,j\in\N:\,|w_{\psi^{j}(1)}|v^{(n)}_{\psi^{j-1}(1)}\leq C v^{(m)}_{\psi^{j}(1)}.$$
		\end{enumerate}
	\end{corollary}
	
	\begin{proof}
		With $\chi:\N\rightarrow\N$ as in the proof of Proposition \ref{conjugacy} it follows for $1\leq p<\infty$
		\begin{eqnarray*}
			\forall\,x\in\omega:\,\sum_{j=1}^\infty\big(|x_j|v^{(m)}_j\big)^p&=&\sum_{j=1}^\infty\big(|\prod_{l=1}^j w_{\chi(l)} x_{\chi(j)}|\frac{v^{(m)}_{\chi(j)}}{\prod_{l=1}^j|w_{\chi(l)}|}\big)^p\\
			&=&\sum_{j=1}^\infty\big(|(T_{w,\psi}\,x)_j|\frac{v^{(m)}_{\chi(j)}}{\prod_{l=1}^j|w_{\chi(l)}|}\big)^p,
		\end{eqnarray*}
		so that $T_{w,\psi}\,x\in\ell_p\big((\frac{v^{(m)}_{\chi(j)}}{\prod_{l=1}^j|w_{\chi(l)}|})_{j\in\N}\big)$ if and only if $x\in\ell_p(v^{(m)})$. Thus, for $1\leq p<\infty$
		$$T_{w,\psi}:k_p(V)\rightarrow k_p(V_{w,\psi})$$
		is a well-defined, continuous bijection (even a stepwise isometric isomorphism), where the decreasing sequence of strictly positive weights $V_{w,\psi}=(v^{(m)}_{w,\psi})_{m\in\N}$ is given by
		$$
		\forall\,m,j\in\N:\,(v^{(m)}_{w,\psi})_j=\frac{v^{(m)}_{\chi(j)}}{\prod_{l=1}^j|w_{\chi(l)}|}=\frac{v^{(m)}_{\psi^{j-1}(1)}}{\prod_{l=1}^j |w_{\psi^{l-1}(1)}|}.
		$$
		Now, the claim follows for $1\leq p<\infty$ directly from Proposition \ref{conjugacy} and Proposition \ref{continuity of the backward shift on Koethe coechelon spaces}. The case $p=0$ is treated analogously.
	\end{proof}
	
	\begin{corollary}\label{dynamics for weighted generalized backward shifts}
		Let $\psi$ be a symbol, $w$ a weight sequence and $V=(v^{(m)})_{m\in\N}$ be a decreasing sequence of strictly positive weights such that $B_{w,\psi}$ is a well-defined, continuous linear operator on $k_p(V)$, $p\in\{0\}\cup[1,\infty)$. Then the following hold.
		\begin{enumerate}
			\item[a)] The following are equivalent.
			\begin{enumerate}
				\item[i)] $B_{w,\psi}$ is transitive on $k_p(V)$.
				\item[ii)] There is $m\in\N$ such that $$\liminf_{j\rightarrow\infty}\frac{v^{(m)}_{\psi^{j}(1)}}{\prod_{l=0}^{j}|w_{\psi^{l}(1)}|}=0.$$
			\end{enumerate}
			\item[b)] The following are equivalent.
			\begin{enumerate}
				\item[i)] $B_{w,\psi}$ is (sequentially) hypercyclic on $k_p(V)$.
				\item[ii)] There are $m\in\N$ and a thick set $I\subseteq\N$ such that
				$$\lim_{I\ni j\rightarrow\infty}\frac{v^{(m)}_{\psi^{j}(1)}}{\prod_{l=0}^{j}|w_{\psi^{l}(1)}|}=0.$$
			\end{enumerate}
			\item[c)] The following are equivalent.
			\begin{enumerate}
				\item[i)] $B$ is mixing on $k_p(V)$.
				\item[ii)] For every infinite set $I\subseteq\N$ there is $m\in\N$ such that
				$$\liminf_{I\ni j\rightarrow \infty}\frac{v^{(m)}_{\psi^{j}(1)}}{\prod_{l=0}^{j}|w_{\psi^{l}(1)}|}=0.$$
			\end{enumerate}
			\item[d)] Assume there is $m\in\N$ such that for every $\varepsilon>0$ the set
			$$I_\varepsilon:=\left\{j\in\N;\,\frac{v^{(m)}_{\psi^{j}(1)}}{\prod_{l=0}^{j}|w_{\psi^{l}(1)}|}<\varepsilon\right\}$$
			is syndetic. Then $B_{w,\psi}$ is topologically ergodic on $k_p(V)$.
			\item[e)] The following are equivalent.
			\begin{enumerate}
				\item[i)] $B_{w,\psi}$ has a periodic point $x\in k_p(V), x\neq 0$.
				\item[ii)] There is $m\in\N$ such that
				$$\left(\frac{v^{(m)}_{\psi^{j}(1)}}{\prod_{l=0}^{j}|w_{\psi^{l}(1)}|}\right)_{j\in\N}\in\ell_p,$$
				respectively,
				$$\left(\frac{v^{(m)}_{\psi^{j}(1)}}{\prod_{l=0}^{j}|w_{\psi^{l}(1)}|}\right)_{j\in\N}\in c_0$$
				when $p=0$.
				\item[iii)] $B_{w,\psi}$ is chaotic on $k_p(V)$.
			\end{enumerate}
		\end{enumerate}
	\end{corollary}
	
	\begin{proof}
		Recall that two continuous self-maps $T:X\rightarrow X$ and $S:Y\rightarrow Y$ on topological spaces $X,Y$ are conjugate if there is a homeomorphism $\phi:X\rightarrow Y$ such that $S\circ \phi=\phi\circ T$. As seen in the proof of Corollary \ref{continuity of generalized backward shifts on Koethe coechelon spaces}, $B_{w,\psi}$ on $k_p(V)$ and $B$ in $k_p(V_{w,\psi})$ are conjugate via $T_{w,\psi}$. Since all considered dynamical properties are preserved under conjugacy and since $x$ is periodic for $B_{w,\psi}$ if and only if $T_{w,\psi}\,x$ is periodic for $B$, the claim follows from Theorem \ref{dynamics of backward shift on Koethe coechelon spaces} and Proposition \ref{chaos for Koethe coechelon spaces}.
	\end{proof}
	
	\section{Examples and open problems}\label{examples}

	In this section we present some examples. We begin by considering topological dual spaces of power series spaces of infinite type.
	
	\subsection{Weighted generalized backward shifts on duals of power series spaces of infinite type}
	
	Let $(\alpha_j)_{j\in\N}$ be an increasing sequence of positive real numbers with $\lim_{j\rightarrow\infty}\alpha_j=\infty$. As in \cite{MeVo1997}, we set with $a^{(m)}_j=e^{m\alpha_j}$
	$$\Lambda_\infty(\alpha):=\lambda_2((a^{(m)})_{m\in\N}):=\{x\in\omega;\,\forall\,m\in\N:\,\|x\|_m^2:=\sum_{j=1}^\infty |x_j|^2e^{2m\alpha_j}<\infty\}.$$
	Then, the topological dual $\Lambda_\infty'(\alpha)$ of $\Lambda_\infty(\alpha)$ is topologically isomorphic to
	\begin{eqnarray*}
		k_2\big((\frac{1}{a^{(m)}})_{m\in\N}\big)&=&k_2\big(\big((e^{-m\alpha_j})_{j\in\N}\big)_{m\in\N}\big)\\
		&=&\{x\in\omega;\,\exists\,m\in\N:\,\sum_{j=1}^\infty |x_j|^2e^{-2m\alpha_j}<\infty\}.
	\end{eqnarray*}
	The particular case of $\alpha_j=j$ gives $\Lambda_\infty((j)_{j\in\N})\cong \mathscr{H}(\C)$, the latter denoting the space of entire functions, and $\Lambda_\infty'((j)_{j\in\N})\cong \mathscr{H}(\{0\})$, the space of germs of holomorphic functions in $0$. Weighted backward shifts on $k_2((j)_{j\in\N})$, i.e.\ weighted generalized backward shifts with $\psi(j)=j+1$, played an important role in the investigation of weighted backward shifts on spaces of real analytic functions in \cite{DoKa18}.
	
	For a symbol $\psi:\N\rightarrow\N\backslash\{1\}$ and a weight sequence $w$ we have by Corollary~\ref{continuity of generalized backward shifts on Koethe coechelon spaces} that $B_{w,\psi}$ is well-defined (and continuous) on $\Lambda_\infty'(\alpha)$ if and only if
	\begin{eqnarray*}
		\forall\,m\in\N \ \ \exists\,n\in\N:\  \infty &>&\sup_{j\in\N}\frac{|w_{\psi^{j}(1)}|e^{-n\alpha_{\psi^{j-1}(1)}}}{e^{-m\alpha_{\psi^{j}(1)}}}\\
		&=& \exp\big(\sup_{j\in\N}(\log|w_{\psi^{j}(1)}|+m\alpha_{\psi^{j}(1)}-n\alpha_{\psi^{j-1}(1)})\big)
	\end{eqnarray*}
	so that $B_{w,\psi}$ is well-defined (and continuous) on $\Lambda_\infty'(\alpha)$ precisely when
	\begin{equation}\label{continuity on dual}
	\forall\,m\in\N \ \ \exists\,n\in\N: \ \ \sup_{j\in\N} \,\big(\log|w_{\psi^{j}(1)}|+m\alpha_{\psi^{j}(1)}-n\alpha_{\psi^{j-1}(1)}\big)<\infty.
	\end{equation}
	In case of $\psi(j)=j+1$ we obtain the weighted backward shift which we simply denote by $B_w$. Since then $\psi^{l}(1)=l+1, l\in\N_0$, by (\ref{continuity on dual}), $B_w$ is well-defined and continuous on $\Lambda_\infty'(\alpha)$ if and only if
	\begin{equation}\label{continuity weighted backward shift on dual}
	\forall\,m\in\N \ \ \exists\,n\in\N: \ \ \sup_{j\in\N}\big(\log|w_{j+1}|+m\alpha_{j+1}-n\alpha_j\big)<\infty.
	\end{equation}
	
	\begin{corollary}
		Let $\psi$ be a symbol and let $\alpha=(\alpha_j)$ be as above such that the generalized backward shift $B_\psi$ is well-defined on $\Lambda_\infty'(\alpha)$. Then $B_\psi$ is mixing on $\Lambda_\infty'(\alpha)$. Moreover, $B_\psi$ is chaotic on $\Lambda_\infty'(\alpha)$ if, and only if, $\sum_{j=1}^\infty e^{-l\alpha_j}<\infty$ for some $l>0$.
	\end{corollary}
	
	\begin{proof}
		Since $\psi$ is a symbol, we have the disjoint union $\N=\{1\}\cup\bigcup_{j\in\N}\psi^{j}(1)$. In particular,
		$$
		\lim_{j\rightarrow\infty}\psi^{j}(1)=\infty,
		$$
		so that with $v^{(m)}_j=e^{-m \alpha_j}, m,j\in\N$, due to the fact that $(\alpha_j)_{j\in\N}$ increases to infinity, we have
		$$
		\forall\,m\in\N: \ \ 0=\lim_{j\rightarrow\infty}e^{-m\alpha_{\psi^{j}(1)}}=\lim_{j\rightarrow\infty} v^{(m)}_{\psi^{j}(1)}.$$
		Because $\Lambda_\infty'(\alpha)=k_2((v^{(m)})_{m\in\N})$, the claim follows from Corollary \ref{dynamics for weighted generalized backward shifts}.
	\end{proof}
	
	\subsection{The annihilation operator on $\mathscr{S}'(\R)$.} The special case of $\alpha=(\log j)_{j\in\N}$ in the previous subsection gives $\Lambda_\infty'(\alpha)=s'$, the space of slowly increasing sequences, i.e.\ the strong dual space of
	\[s:=\lambda_2\left(((j^m)_{j\in\N})_{m\in\N}\right)=\{x\in\omega;\,\forall\,m\in\N:\,\|x\|_m^2:=\sum_{j=1}^\infty |x_j|^2 j^{2m}<\infty\}.\]
	It follows from (\ref{continuity weighted backward shift on dual}) that the weighted backward shift $B_w$ with weight sequence $w$ is a well-defined and continuous operator on $s'$ if and only if
	$$\forall\,m\in\N\,\exists\,n\in\N:\,\sup_{j\in\N}\frac{|w_{j+1}|(j+1)^m}{j^n}<\infty.$$
	Given a weight sequence $w$ satisfying the above condition, it follows that the weighted backward shift $B_w$ is transitive, hypercyclic, etc.\ if (and only if) there is $m\in\N$ such that the sequence
	$$\left(\frac{1}{j^m\prod_{l=1}^j |w_l|}\right)_{j\in\N}$$
	satisfies the respective properties mentioned in part a), b) etc.\ of Corollary \ref{dynamics for weighted generalized backward shifts}.
	
	Instead of repeating these conditions explicitly, we just consider the special weighted backward shift $B_w$ with weight sequence $w_j=\sqrt{j}$. By the above, $B_{(\sqrt{j})_{j\in\N}}$ is clearly well-defined and continuous on $s'$. As is well-known, see e.g.\ \cite[Example 29.5(2)]{MeVo1997}, via Hermite expansion, $B_{(\sqrt{j})_{j\in\N}}$ on $s'$ is conjugate to the annihilation operator on $\mathscr{S}'(\R)$ (when the latter is equipped with the strong dual topology), i.e.\ to the operator
	$$A_{-}:\mathscr{S}'(\R)\rightarrow\mathscr{S}'(\R), u\mapsto \frac{1}{\sqrt{2}}(u'+xu),$$
	where we denote the multiplication operator with the identity on $\mathscr{S}'(\R)$ simply by $u\mapsto xu, u\in\mathscr{S}'(\R)$. Dynamical properties of the annihilation operator on the Fr\'echet space $s$ of rapidly decreasing sequences were studied in \cite{GuMC1996} and taken up on a different Fr\'echet space in \cite{DeEmIn2000}.
	
	\begin{corollary}
		The annihilation operator
		$$A_{-}:\mathscr{S}'(\R)\rightarrow\mathscr{S}'(\R), u\mapsto \frac{1}{\sqrt{2}}(u'+xu)$$
		on $\mathscr{S}'(\R)$ equipped with the strong dual topology is mixing, sequentially hypercyclic, topologically ergodic, and chaotic.
	\end{corollary}
	
	\begin{proof}
		It follows immediately from
		$$\forall\,m\in\N:\,\big(\frac{1}{j^m\sqrt{j!}}\big)_{j\in\N}\in\ell_2$$
		and Corollary \ref{dynamics for weighted generalized backward shifts} that $B_{(\sqrt{j})_{j\in\N}}$ is mixing, (sequentially) hypercyclic, chaotic, and topologically ergodic on $k_2\big(\big((\frac{1}{j^m})_{j\in\N}\big)_{m\in\N}\big)=s'$. Hence, the claim follows by conjugacy.
	\end{proof}
	
	\subsection{Separating examples.}
	
	In this subsection we provide examples of K\"othe coechelon spaces such that the backward shift $B$ is well-defined and continuous on these spaces as well as topologically ergodic but not hypercyclic, mixing but does not satisfy the sufficient condition for topological ergodicity from Theorem \ref{dynamics of backward shift on Koethe coechelon spaces} d), respectively. Moreover, we give an example of a nuclear K\"othe coechelon space on which the backward shift is transitive but not hypercyclic.
	
	\begin{proposition}\label{topological ergodic but not hypercyclic}
		There is a decreasing sequence $V=(v^{(m)})_{m\in\N}$ of strictly positive weights such that $B$ on $k_p(V)$ is topologically ergodic but not hypercyclic.
	\end{proposition}
	
	\begin{proof}
		We set
		$$v^{(1)}_j=2^{-n}\text{ if }j=2^{n-1}(2k-1)\text{ for some }k,n\in\N.$$
		With this, we define recursively
		$$\forall\,m,j\in\N:\,v^{(m+1)}_j:=\min\{v^{(m)}_j,v^{(m)}_{j+1}\}$$
		so that $V=(v^{(m)})_{m\in\N}$ is a decreasing sequence of strictly positive weights such that $B$ is well-defined and continuous on $k_2(V)$ by Proposition \ref{continuity of the backward shift on Koethe coechelon spaces}.
		
		Clearly, $(v^{(m)})_{m\in\N}$ satisfies the condition in Theorem \ref{dynamics of backward shift on Koethe coechelon spaces} d) (with $m=1$) but the condition under b) is not fulfilled so that $B$ is topologically ergodic on $k_p(V)$ but not hypercyclic.
	\end{proof}
	
	\begin{proposition}\label{nuclear exmaple}
		There is a decreasing sequence $V=(v^{(m)})_{m\in\N}$ of strictly positive weights such that $k_2(V)$ is nuclear and the backward shift $B$ is transitive on $k_2(V)$ but not hypercyclic.
	\end{proposition}
	
	\begin{proof}
		For every $j\in\N$ there are unique $n(j)\in\N$ and $r(j)\in\{0\,\ldots,2^{n(j)-1}-1\}$ such that $j=2^{n(j)}-r(j)$. Then,
		$$\forall\,j\in\N:\,2^{n(j+1)}-r(j+1)=j+1=2^{n(j)}-r(j)+1,$$
		so that either
		$$n(j)=n(j+1)\text{ and }r(j+1)=r(j)-1$$
		or
		$$n(j+1)=n(j)+1\text{ and }r(j)=0, r(j+1)=2^{n(j)}-1.$$
		
		For $m,j\in\N$ we define
		$$v^{(m)}_j:=\begin{cases}
		\frac{1}{2^{n(j)}j^{2m}},&\text{ if } r(j)<m,\\ \frac{2^j}{j^{2m}},& \text{else.}
		\end{cases}$$
		Then, for fixed $j\in\N$ the sequence $(v^{(m)}_j)_{m\in\N}$ is decreasing.
		
		We first show that for each $m\in\N$ there is $C>0$ such that $v^{(m+1)}_j\leq C v^{(m)}_{j+1}$ for every $j\in\N$ so that the backward shift $B$ is well-defined (and continuous) on $k_2(V)$ by Proposition \ref{continuity of the backward shift on Koethe coechelon spaces}.
		
		So, we fix $m\in\N$. For $j\in\N$ we consider first the case that $r(j)<m+1$. If additionally $n(j)=n(j+1)$ we have $r(j+1)=r(j)-1<m$ so that
		$$\frac{v^{(m+1)}_j}{v^{(m)}_{j+1}}=\frac{1/(2^{n(j)} j^{2(m+1)})}{1/(2^{n(j+1)}(j+1)^{2m})}=\frac{1}{j^2}(1+\frac{1}{j})^{2m}\leq 4^m.$$
		On the other hand, if $n(j+1)=n(j)+1$ we have $r(j)=0$ and $r(j+1)=2^{n(j)}-1$ so that
		$$\frac{v^{(m+1)}_j}{v^{(m)}_{j+1}}=\begin{cases}r(j+1)\geq m:&\dfrac{1/(2^{n(j)}j^{2(m+1)})}{2^{j+1}/(j+1)^{2m}}=\frac{(1+\frac{1}{j})^{2m}}{j^2 2^{n(j)+j+1}}\leq 4^m\\
		&\\
		r(j+1)<m:&\dfrac{1/(2^{n(j)}j^{2(m+1)})}{1/(2^{n(j+1)}(j+1)^{2m})}=\frac{2(1+\frac{1}{j})^{2m}}{j^2}\leq 4^{m+1}.\end{cases}$$
		Thus in case $r(j)<m+1$ we have $v^{(m+1)}_j\leq 4^{m+1} v^{(m)}_{j+1}$.
		
		In case of $r(j)\geq m+1$ we have in particular $r(j)>0$ so that $n(j)=n(j+1)$ as well as $r(j+1)=r(j)-1\geq m$ hold. Then
		$$\frac{v^{(m+1)}_j}{v^{(m)}_{j+1}}=\frac{2^j/j^{2(m+1)}}{2^{j+1}/(j+1)^{2m}}=\frac{1}{2 j^2}(1+\frac{1}{j})^{2m}\leq 4^m.$$
		Hence, we have shown $v^{(m+1)}_j\leq 4^{m+1} v^{(m)}_{j+1}$ for all $j\in\N$. It should be noted that for fixed $m\in\N$ and $j=2^n-m, n\in\N,$ we have $j+1=2^n-(m-1)$, i.e.\ $n(j)=n(j+1)=n$, $r(j)=m$, $r(j+1)=m-1$ so that
		$$\frac{v^{(m)}_j}{v^{(m)}_{j+1}}=\frac{2^j/j^{2m}}{1/(2^{n(j+1)}(j+1)^{2m})}=2^{j+n(j+1)}(1+\frac{1}{j})^{2m}\geq 2^{2^n-m+n}$$
		so that $\sup_{j\in\N}v^{(m)}_j/v^{(m)}_{j+1}=\infty$ and thus $\ell_2(v^{(m)})$ is not $B$-invariant.
		
		Obviously, for every $m\in\N$ it holds
		$$\liminf_{j\rightarrow\infty}v^{(m)}_j=0$$
		so that $B$ is transitive on $k_2(V)$ by Theorem \ref{dynamics of backward shift on Koethe coechelon spaces} a). Moreover, because
		$$\forall\,m\in\N:\,\lim_{j\rightarrow\infty}\frac{2^j}{j^{2m}}=\infty$$
		it follows
		$$\forall\,m\in\N, \varepsilon>0\,\exists L\in\N\,\forall\,l\geq L, r\geq m:\,\sup_{1\leq j\leq r}v^{(m)}_{l+j}>\varepsilon$$
		so that by Theorem \ref{dynamics of backward shift on Koethe coechelon spaces} b) $B$ is not hypercyclic on $k_2(V)$.
		
		Finally, due to
		$$\frac{v^{(m+1)}_j}{v^{(m)}_j}=\begin{cases}
		r(j)<m:&\frac{1/(2^{n(j)} j^{2(m+1)})}{1/(2^{n(j)} j^{2m})}=\frac{1}{j^2}\\
		&\\
		r(j)=m:&\frac{1/(2^{n(j)} j^{2(m+1)})}{2^j/j^{2m}}=\frac{1}{j^2}\frac{1}{2^{n(j)+j}}\\
		&\\
		r(j)\geq m+1:&\frac{2^j/j^{2(m+1)}}{2^j/j^{2m}}=\frac{1}{j^2}
		\end{cases}$$
		it follows that $(v^{(m+1)}_j/v^{(m)}_j)_{j\in\N}\in\ell_1$ so that $k_2(V)$ is nuclear (cf.\ \cite[Proposition 2.15]{Bi1988}).
	\end{proof}
	
	\begin{proposition}\label{mixing but not topologically ergodic condition}
		There is a decreasing sequence $V=(v^{(m)})_{m\in\N}$ of strictly positive weights such that $B$ is mixing on $k_p(V)$ but does not satisfy the sufficient condition for topological ergodicity from Theorem \ref{dynamics of backward shift on Koethe coechelon spaces} d).
	\end{proposition}
	
	\begin{proof}
		The following construction is inspired by the example \cite[Example 27.21]{MeVo1997} of a K\"othe echelon space which is a Montel space but not a Schwartz space. We fix a bijection $\varphi:\N^2\rightarrow\N$ such that, for every $k\in\N$ there are arbitrarily long bounded intervals $I$ with
		$$I\cap\N=\varphi\big(\{(l,k);\,l\in F\}\big)$$
		for some finite set $F\subseteq\N$. We then define $\hat{v}^{(1)}_j:=1, j\in\N$ and for $m,j\in\N, m\geq 2$
		$$\hat{v}^{(m)}_j:=\begin{cases}
		(m l)^{-m},&\text{ if }j=\varphi(l,k)\text{ with }k<m\\ m^{-k},&\text{ if }j=\varphi(l,k)\text{ with }k\geq m.
		\end{cases}$$
		By the choice of $\varphi$ it follows that for every $m\in\N$ and every $\varepsilon\in(0,1/m^m)$ there are arbitrarily long bounded intervals $I$ such that $\hat{v}^{(m)}_j>\varepsilon$ for every $j\in I\cap\N$.
		
		To have a decreasing sequence $V=(v^{(m)})_{m\in\N}$ of strictly positive weights such that $B$ is continuous on $k_p(V)$, we set $v^{(1)}_j:=\hat{v}^{(1)}_j, j\in\N$ as well as
		$$\forall\,m,j\in\N:\,v^{(m+1)}_j:=\min\{\hat{v}^{(m+1)}_j, v^{(m)}_j, v^{(m)}_{j+1}\}.$$
		By induction on $m$ it follows that for every $m\in\N$ and every $\varepsilon\in(0, 1/m^m)$ there are arbitrarily long bounded intervals $I$ such that $v^{(m)}_j>\varepsilon$ for every $j\in I\cap\N$. Hence, $V=(v^{(m)})_{m\in\N}$ does not satisfy the sufficient condition for topological ergodicity of $B$ from Theorem \ref{dynamics of backward shift on Koethe coechelon spaces} d).
		
		On the other hand, let $I\subseteq\N$ be infinite. In case there is a finite $F\subseteq\N$ with $I\subseteq\varphi(\N\times F)$ we select $m\in\N$ such that $m>k$ for every $k\in F$ so that
		$$v^{(m)}_j\leq\hat{v}^{(m)}_j=(m l)^{-m}\text{ if }j=\varphi(l,k),$$
		i.e.\ $\inf_{j\in I}v^{(m)}_j=0$. In case that there is no finite $F\subseteq\N$ with $I\subseteq\varphi(\N\times F)$ there are sequences of natural numbers $(l_n)_{n\in\N}, (k_n)_{n\in\N}$, where $(k_n)_{n\in\N}$ is strictly increasing, such that $\varphi(l_n,k_n)\in I$ for every $n\in\N$. Since $v^{(2)}_{\varphi(l_n,k_n)}\leq 2^{-k_n}$ we obtain $\inf_{j\in I}v^{(2)}_j=0$. Hence, by Theorem \ref{dynamics of backward shift on Koethe coechelon spaces} c) $B$ is mixing on $k_p(V)$.
	\end{proof}
	
	\subsection{Snake shift operators}
	
	(This subsection is only contained in the arXiv-Version of the article and not in the published version!) Our last example concerns a construction of (sequentially) hypercyclic operators on direct sums of Fr\'echet spaces given in \cite{BoFrPeWe2005}. In order to fit these kind of operators in our frame, we need to generalize our K\"othe spaces to certain sequence LF-spaces (see, e.g., \cite{BiBo94} and \cite{Vogt1992}).
	
	Unlike to the previous sections, in this subsection our sequences will be indexed by $\N\times\N$ instead of $\N$. We define the inductive limit $E=\ind_n \lambda_p(V^n)$ of K\"othe echelon spaces, where $p\in [1,+\infty]$, the sequence of weights $V^n=(v^{(n,k)})_{k\in\N}$ is so that $v^{(n,k)}_{i,j}\in ]0,+\infty]$ for all $n,k,i,j\in\N$,
	\[
	v^{(n,k)}\leq v^{(n,k+1)} \ \ \mbox{ and } \ \ v^{(n,k)}\geq v^{(n+1,k)} \ \ \forall n,k\in\N ,
	\]
	\[
	\lambda_p(V^n)=\{ x=(x_{i,j})_{i,j}\in\K^{\N\times\N} \ ; \ (x_{i,j}v^{(n,k)}_{i,j})_{i,j}\in\ell_p(\N\times\N) \ \forall k\in\N\},
	\]
	endowed with the increasing sequence of norms $\| x\|_k:=\| (x_{i,j}v^{(n,k)}_{i,j})_{i,j}\|_{\ell_p}$, $k\in\N$. Observe that, since we allow $v^{(n,k)}_{i,j}=\infty$, this means that, if $x=(x_{i,j})_{i,j}\in\lambda_p(V^n)$, then $x_{i,j}=0$. These spaces contain direct sums of classical sequence spaces like $\oplus_n\ell_p$ and $\oplus_n s$, where $s$ is the space of rapidly decreasing sequences. Actually,
	\[
	E=\oplus_n\ell_p=\ind_n\lambda_p(V^n), \ \mbox{ where } \ v_{i,j}^{(n,k)}=1 \ \mbox{ if } \ i\leq n, \ \forall j,k \in\N,
	\]
	\[
	\mbox{ and } \ v_{i,j}^{(n,k)}=\infty \ \mbox{ if } \ i>n, \ \ \forall j,k \in\N,
	\]
	and
	\[
	E=\oplus_n s=\ind_n\lambda_1(\tilde{V}^n), \ \mbox{ where } \ \tilde{v}_{i,j}^{(n,k)}=j^k \ \mbox{ if } \ i\leq n, \ \forall j,k \in\N,
	\]
	\[
	\mbox{ and } \ \tilde{v}_{i,j}^{(n,k)}=\infty \ \mbox{ if } \ i>n, \ \ \forall j,k \in\N.
	\]
	Given a symbol $\psi:\N\times\N\to\N\times\N\setminus\{(1,1)\}$ (which is defined analogously as in the case of $\N$ as index set) and a weight sequence $w=(w_{i,j})_{i,j}$, we consider the weighted generalized backward shift $B_{w,\psi}:\K^{\N\times\N}\to\K^{\N\times\N}$. We are interested in the particular case that the weight sequence is constant, $w_{i,j}=\lambda$ with $|\lambda|>1$ (i.e., $B_{w,\psi}=\lambda B_{\psi}$), and $\psi:\N\times\N\to\N\times\N\setminus\{(1,1)\}$ is defined by
	\[
	\begin{array}{ll}
	(1) \ \ \psi(1,j)=(1,j+1),          &   n(k)<j<n(k+1), \\
	(2) \ \ \psi(2,2k-1)=(1,n(k)+1),    &                  \\
	(3) \ \ \psi(1,n(k+1))=(2,2k),      &                  \\
	(4) \ \ \psi(2k-1,1)=(2k,1),        &                  \\
	(5) \ \ \psi(i,j)=(i+1,j-1),        &   \mbox{ if } i+j \mbox{ even, } i,j>1, \\
	(6) \ \ \psi(i,j)=(i-1,j+1),        &   \mbox{ if } i+j \mbox{ odd, } i>2,
	\end{array}
	\]
	for every $k\in\N$, where $(n(k))_k$ is a suitable sequence such that $(n(k+1)-n(k))_k$ increases and tends to infinity.
	
	This kind of construction was called ``snake shift'' in \cite{BoFrPeWe2005}, and was applied to the spaces $\oplus_n\ell_p$ and $\oplus_n s$. Proceeding as in Corollary~\ref{continuity of generalized backward shifts on Koethe coechelon spaces}, $B_\psi:E\to E$ is (well-defined and) continuous if, and only if,
	\[
	\forall m\in\N, \ \exists n\in\N, \ \forall k\in\N, \ \exists l\in\N \ \mbox{ and } \ C>0 \ \mbox{ with } \ v^{(n,k)}_{r_{j-1}}\leq Cv^{(m,l)}_{r_j}, \ \forall j\in\N,
	\]
	where $r_i=\psi^i(1,1)$, $i\in\N_0$. The selection of $\psi$ ensures the continuity of $B_\psi$ on $\oplus_n\ell_p$ if, e.g., for each $m\in\N$ we take $n=m+1$ in the characterization above. The case $E=\oplus_n s$ is more subtle since condition $(3)$ in the selection of $\psi$ forces that continuity of $B_\psi$ needs that $(n(k))_k$ is polynomially bounded. Actually, it was shown in \cite{BoFrPeWe2005} that we can select $(n(k))_k$ so that $n(k)\leq 3k^2$, $k\in\N$. This means that, for the continuity condition above, for each $m\in \N$, we take $n=m+1$, and for any $k\in \N$, we may choose $l=k^3$.
	
	Concerning the dynamics of snake shifts as defined, and following the argument of Corollary~\ref{dynamics for weighted generalized backward shifts}, it is easy to see that
	$\lambda B_{\psi}$ is (sequentially) hypercyclic on the inductive limit $E=\ind_n \lambda_p(V^n)$ of K\"othe echelon spaces as defined above if, and only if, there are $m\in\N$ and a thick set $I\subseteq\N$ such that, for any $k\in \N$,
	\begin{equation}\label{seq_hc_kothe_lf}			
	\lim_{I\ni j\rightarrow\infty}\frac{v^{(m,k)}_{\psi^{j}(1,1)}}{|\lambda|^j}=0.
	\end{equation}
	To prove the sequential hypercyclicity of $\lambda B_{\psi}$ in our spaces, we fix $m=1$, select the increasing sequences $(j_k)_k$ and $(l_k)_k$, $j_k<l_k$, with
	\begin{equation*}
	\psi^{j_k}(1,1)=(1,n(k)+1) \ \ \mbox{ and } \ \ \psi^{l_k}(1,1)=(1,n(k+1)), \ \ k\in\N.
	\end{equation*}
	By condition (1) in the selection of $\psi$, $\psi^{j}(1,1)=(1,n(k)+j-j_k+1)$ for $j_k\leq j\leq l_k$, $l_k-j_k=n(k+1)-n(k)-1$ and, since $(n(k+1)-n(k))_k$ tends to infinity, we have that $I:=\bigcup_k[j_k,l_k]\cap\N$ is a thick set. When $E=\oplus_n\ell_p$, $v^{(1,k)}_{1,j}=1$ for all $j,k\in\N$, and condition~\eqref{seq_hc_kothe_lf} is trivially satisfied. In the case $E=\oplus_n s$, we observe that $\psi^j(1,1)=(1,p_j)$, where $p_j<j$, for every $j\in I$. Therefore,
	\[			
	\lim_{I\ni j\rightarrow\infty}\frac{\tilde{v}^{(1,k)}_{\psi^{j}(1,1)}}{|\lambda|^j}=\lim_{I\ni j\rightarrow\infty}\frac{\tilde{v}^{(1,k)}_{1,p_j}}{|\lambda|^j}=\lim_{I\ni j\rightarrow\infty}\frac{p_j^k}{|\lambda|^j}\leq \lim_{I\ni j\rightarrow\infty}\frac{j^k}{|\lambda|^j}=0,
	\]
	and we conclude the result.

	\begin{remark}
		Most of our main results can be generalized to bilateral shifts on sequence LF-spaces over $\Z$, and to certain weighted composition operators on more general function  LF-spaces. These results will be presented in a forthcoming paper.
	\end{remark}
	
	\subsection{Open problems}
	
	In this final subsection we mention the following natural questions which arise from our results.
	
	(1) In Proposition \ref{ergodic on LF} we gave a sufficient condition for topological ergodicity. However, this condition is not necessary by Proposition \ref{mixing but not topologically ergodic condition}. Which condition completely characterizes topological ergodicity of $B$?
	
	(2) Is there a nuclear K\"othe coechelon space $k_p(V)$ on which the backward shift $B$ is topologically ergodic but not sequentially hypercyclic? Such an example would be a strengthening of both Propositions \ref{topological ergodic but not hypercyclic} and \ref{nuclear exmaple}.

	\section*{Acknowledgements}
	
	We would like to thank the anonymous referee for reading the article with great
	care and detail, and for making valuable suggestions which in particular led to
	Proposition~\ref{general transitivity} that extends our Corollary~\ref{transitivity on LF}. Moreover, we are also thankful for pointing out references \cite{GuMC1996} and \cite{DeEmIn2000}.
	
	The research on the subject of this paper was initiated during a stay of the second named author at Universitat Polit\`ecnica de Val\`encia. He expresses his deep gratitude towards his coauthors for the invitation as well as for the warm hospitality during this stay as well as during many earlier visits.
	
	The first author was  partially supported by Projects MTM2016-76647-P and GV Prometeo 2017/102. The research of the third author was partially supported by Projects  MTM2016-75963-P and GV Prometeo 2017/102.


\begin{thebibliography}{10}
		\parskip=0.7pt
		
	\bibitem{BaMa09}
	Bayart F.\ and Matheron \'{E}.
	\newblock \emph{Dynamics of linear operators}, volume 179 of \emph{Cambridge
		Tracts in Mathematics}.
	\newblock Cambridge University Press, Cambridge, 2009.
	
	\bibitem{BeMePePu19}
	B\`es, J., Menet, Q., Peris, A. and Puig, Y.
	\newblock Strong transitivity properties for operators.
	\newblock \emph{J. Differential Equations} 266(2-3):1313--1337, 2019.
	
	\bibitem{Bi1988}
	Bierstedt, K.-D.
	\newblock An introduction to locally convex inductive limits.
	\newblock In \emph{Functional analysis and its Applications
		(Nice, 1986)}, World Sci. Publ. Singapore, pages 33--135. 1988.
	
	\bibitem{BiBo94}
	Bierstedt, K.D. and Bonet, J.
	\newblock Weighted (LF)-spaces of continuous functions.
	\newblock \emph{ Math. Nachr.} 165:25--48, 1994.
	
	\bibitem{BiBo03}
	Bierstedt, K.D. and Bonet, J..
	\newblock Some aspects of the modern theory of Fr\'echet spaces.
	\newblock \emph{Rev. R. Acad. Cienc. Exactas F\'is. Nat. Ser. A Mat. RACSAM} 97(2):159--188, 2003.
	
	\bibitem{BiMeSu1982}
	Bierstedt, K.D., Meise, R.-G. and Summers, W.-H.
	\newblock K\"{o}the sets and {K}\"{o}the sequence spaces.
	\newblock In \emph{Functional analysis, holomorphy and approximation theory
		({R}io de {J}aneiro, 1980)}, volume~71 of \emph{North-Holland Math. Stud.},
	pages 27--91. North-Holland, Amsterdam-New York, 1982.
	
	\bibitem{Bo00}
	Bonet, J.
	\newblock Hypercyclic and chaotic convolution operators.
	\newblock \emph{J. London Math. Soc.}, 62:253--262, 2000.
	
	\bibitem{BoFrPeWe2005}
	Bonet, J., Frerick, L., Peris, A. and Wengenroth, J.
	\newblock Transitive and hypercyclic operators on locally convex spaces.
	\newblock \emph{Bull. London Math. Soc.}, 37(2):254--264, 2005.
	
	\bibitem{BoDo12}
	Bonet, J. and Doma\'{n}ski, P.
	\newblock Hypercyclic composition operators on spaces of real analytic functions.
	\newblock \emph{Math. Proc. Cambridge Philos. Soc.}, 153(3):489--503, 2012.
	
	\bibitem{DeEmIn2000}
	Decarreau, A., Emamirad, E. and Intissar, A.
	\newblock Chaoticit\'e de l'op\'erateur de Gribov dans l'espace de Bargmann.
	\newblock \emph{C. R. Acad. Sci., Paris, S\'er. I, Math.} 331(9):751--756, 2000.
	
	
	\bibitem{DiDo1993}
	Dierolf, S. and Doma\'{n}ski, P.
	\newblock Factorization of {M}ontel operators.
	\newblock \emph{Studia Math.}, 107(1):15--32, 1993.
	
	\bibitem{DoKa18}
	Doma\'{n}ski, P. and Kar{\i}ks{\i}z, C.-D.
	\newblock Eigenvalues and dynamical properties of weighted backward shifts on
	the space of real analytic functions.
	\newblock \emph{Studia Math.}, 242(1):57--78, 2018.
	
	\bibitem{Grosse-Erdmann2000}
	Grosse-Erdmann, K.-G.
	\newblock Hypercyclic and chaotic weighted shifts.
	\newblock \emph{Studia Math.}, 139(1):47--68, 2000.
	
	\bibitem{GEPe10}
	Grosse-Erdmann, K.-G. and Peris, A.
	\newblock Weakly mixing operators on topological vector spaces.
	\newblock \emph{Rev. R. Acad. Cienc. Exactas F\'is. Nat. Ser. A Mat. RACSAM} 104(2):413--426, 2010.
	
	\bibitem{GEPe11}
	Grosse-Erdmann, K.-G. and Peris-Manguillot, A.
	\newblock \emph{Linear chaos}.
	\newblock Universitext. Springer, London, 2011.
	
	\bibitem{GuMC1996}
	Gulisashvili, A. and MacCluer, C.-R.
	\newblock Linear chaos in the unforced quantum harmonic oscillator.
	\newblock \emph{J. Dyn. Syst. Meas. Control}, 118(2):337--338, 1996.
	
	\bibitem{Kalmes19-3}
	Kalmes, T.
	\newblock Dynamics of weighted composition operators on function spaces defined
	by local properties.
	\newblock \emph{Studia Math.}, 249(3):259--301, 2019.
	
	\bibitem{MaPe2002}
	Mart\'{\i}nez-Gim\'{e}nez, F. and Peris, A.
	\newblock Chaos for backward shift operators.
	\newblock \emph{Internat. J. Bifur. Chaos Appl. Sci. Engrg.}, 12(8):1703--1715,
	2002.
	
	\bibitem{MeVo1997}
	Meise, R. and Vogt, D.
	\newblock \emph{Introduction to functional analysis}, volume~2 of \emph{Oxford
		Graduate Texts in Mathematics}.
	\newblock The Clarendon Press Oxford University Press, New York, 1997.
	\newblock Translated from the German by M. S. Ramanujan and revised by the
	authors.
	
	\bibitem{MuPe15}
	Murillo-Arcila, M. and Peris, A.
	\newblock Chaotic behaviour on invariant sets of linear operators.
	\newblock \emph{Integral Equations Operator Theory} 81(4):483--497, 2015.
	
	\bibitem{Peris18}
	Peris, A.
	\newblock A hypercyclicity criterion for non-metrizable topological vector
	spaces.
	\newblock \emph{Funct. Approx. Comment. Math.}, 59(2):279--284, 2018.
	
	\bibitem{Salas1995}
	Salas, H.-N.
	\newblock Hypercyclic weighted shifts.
	\newblock \emph{Trans. Amer. Math. Soc.}, 347(3):993--1004, 1995.
	
	\bibitem{Shkarin2012}
	Shkarin, S.
	\newblock Hypercyclic operators on topological vector spaces.
	\newblock \emph{J. Lond. Math. Soc. (2)}, 86(1):195--213, 2012.
	
	\bibitem{Vald1982}
	Valdivia, M.
	\newblock \emph{Topics in Locally Convex Spaces}, volume 67 of \emph{North Holland Math.
		Stud.},
	\newblock Amsterdam, 1983.  
	
	\bibitem{Vogt1983}
	Vogt, D.
	\newblock Sequence space representations of spaces of test functions and distributions.
	\newblock In \emph{Functional analysis, holomorphy and approximation theory
		({R}io de {J}aneiro, 1979)}, volume~83 of \emph{Lecture Notes in Pure and Appl. Math.},
	pages 405--443. Dekker, New York, 1983.
	
	\bibitem{Vogt1992}
	Vogt, D.
	\newblock Regularity properties of (LF)-spaces.
	\newblock In \emph{Progress in Functional Analysis
		(Pe\~n\'{\i}scola, 1990)}, volume~170 of \emph{North-Holland Math. Stud.},
	pages 57--84. North-Holland, Amsterdam, 1992.
	
	\bibitem{Weng03}
	Wengenroth, D.
	\newblock \emph{Derived Functors on Functional Analysis}, volume 1810 of \emph{Lecture Notes in Mathematics}.
	\newblock Springer, Berlin, 2003.		
\end{thebibliography}
\end{document}